\documentclass{gtpart}
\usepackage{amsmath,amssymb}

\numberwithin{equation}{subsection}
\theoremstyle{plain}
\newtheorem{theorem}{Theorem}

\newtheorem{prop}[theorem]{Proposition}
\newtheorem{lem}[theorem]{Lemma}

\newtheorem{cor}[theorem]{Corollary}

\theoremstyle{remark}

\theoremstyle{definition}

\newcommand{\nbd}[2]{\mathcal{N}_{#2} ({#1}) } 

\newcommand{\F}{{\mathcal F}}

\DeclareMathOperator{\diam}{diam}

\newcommand{\acts}{\curvearrowright}
\DeclareMathOperator{\CAT}{CAT}
\DeclareMathOperator{\Cone}{Cone}

\DeclareMathOperator{\Stab}{Stab}

\newcommand{\boundary}{\partial}

\def\Ga{\Gamma}

\newcommand{\oa}{\overrightarrow}

\def\om{\omega}

\def\Si{\Sigma}

\def\ulim{\mathop{\hbox{\textup{$\om$-lim}}}}

\def\cangle{\widetilde\angle} 
\def\tangle{\angle_{T}}       
\def\Td{d_T}                  
\def\tits{\partial_T}         

\def\RomanianComma#1{\setbox0=\hbox{#1}{\ooalign{\hidewidth
    \lower1.2ex\hbox{$\mspace{1mu}^{\text{,}}$}\hidewidth\crcr\unhbox0}}}
\newcommand{\Drutu}{Dru{\RomanianComma{t}u}}

\usepackage{ifthen}

\newcommand{\showcomments}{yes}

\newsavebox{\commentbox}
{\ifthenelse{\equal{\showcomments}{yes}}%
{\footnotemark
    \begin{lrbox}{\commentbox}
    \begin{minipage}[t]{1.25in}\raggedright\sffamily\upshape\tiny
    \footnotemark[\arabic{footnote}]}
{\begin{lrbox}{\commentbox}}}%
{\ifthenelse{\equal{\showcomments}{yes}}%
{\end{minipage}\end{lrbox}\marginpar{\usebox{\commentbox}}}
{\end{lrbox}}}

\title{Erratum on ``Hadamard spaces with isolated flats''}

\author{G.~Christopher Hruska}
\givenname{G.~Christopher}
\surname{Hruska}
\address{Department of Mathematical Sciences \\
         University of Wisconsin--Milwaukee \\
	 PO Box 413 \\
	 Milwaukee, WI 53201-0413, USA}
\email{chruska@uwm.edu}
\urladdr{http://www.uwm.edu/~chruska}

\author{Bruce Kleiner}
\givenname{Bruce}
\surname{Kleiner}
\address{Department of Mathematics\\
Yale University\\
PO Box 208283\\
New Haven, CT 06520-8283}
\email{bruce.kleiner@yale.edu}
\urladdr{http://www.math.yale.edu/~bk255}

\keyword{isolated flats}
\keyword{asymptotic cone}
\keyword{relative hyperbolicity}

\subject{primary}{msc2000}{20F67}
\subject{secondary}{msc2000}{20F69}

\begin{document}

\begin{abstract}
The purpose of this erratum is to correct the proof of
Theorem~A.0.1 in the appendix to \cite{HK05},
which was jointly authored by Mohamad Hindawi, Hruska and Kleiner.
In that appendix, many of the results of \cite{HK05}
about $\CAT(0)$ spaces with isolated flats
are extended to a more general setting in which the isolated subspaces
are not necessarily flats.
However, one step of that extension does not follow from the argument
we used the isolated flats setting.
We provide a new proof that fills this gap.

In addition, we give a more detailed account of several other parts
of Theorem~A.0.1, which were sketched in \cite{HK05}.
\end{abstract}

\maketitle

\section{Introduction}
\label{sec:Introduction}

The purpose of this erratum is to explain a gap in the proof of
\cite[Theorem~A.0.1]{HK05} and to explain how to fill it.
The arguments used to fill this gap use ideas not present in \cite{HK05}.

In addition, we present details for several other portions of the proof
of \cite[Theorem~A.0.1]{HK05},
which were briefly sketched in \cite{HK05}.
The new details for these portions
are easy modifications of arguments from \cite{HK05}.
Nevertheless, the exact nature of these modifications
was, perhaps, not described as explicitly as it could have been.

The main results of \cite{HK05} explain the structure of
a $\CAT(0)$ space $X$ with isolated flats and the structure
of a group $\Gamma$ acting properly, cocompactly and isometrically
on such a space.
In a short appendix, written jointly by the authors and Mohamad Hindawi,
we extend those results to a more general setting in which the
isolated subspaces are not necessarily flats.
Many of the details of this extension are nearly identical
to the details presented for isolated flats.
As a result, the details of the extension were not given explicitly.

However, one step of the proof of \cite[Theorem~A.0.1]{HK05}
does not follow from the same
reasoning given in the isolated flats case.
Specifically, the proof of \cite[Lemma~3.3.1]{HK05}
is correct in the setting of isolated flats
but does not extend directly to the more general setting
required by the appendix.
The gap occurs at the point where we prove that two flats $F,F'$
in the asymptotic cone obtained as ultralimits of flats in $X$
cannot intersect in more than one point.
The argument assumes the existence of many nondegenerate triangles
in $F'$.  In particular, if $x,y \in F \cap F'$
we use that the set of points $z \in F'$ with 
$\Delta(x,y,z)$ nondegenerate is a dense set of $F'$.
This conclusion is certainly true in the isolated flats case,
since the only degenerate triangles in a flat
are those for which $x$, $y$ and $z$
are colinear.
However, this fact is not necessarily true in the more general setting
of the appendix.  For example,
if the isolated subspaces of $X$ are $\delta$--hyperbolic,
their ultralimits are trees, which do not contain nondegenerate
triangles.

We fill this gap by proving Proposition~\ref{prop:AtMostOnePoint},
which states that two different ultralimits of isolated subspaces
cannot intersect in more than one point.
As mentioned above, we also give a more detailed account
of several other parts of the proof of \cite[Theorem~A.0.1]{HK05}.


We begin by recalling the statement of \cite[Theorem~A.0.1]{HK05}.

\begin{theorem}[Theorem~A.0.1, \cite{HK05}]
\label{thm:Main}
Let $X$ be a $\CAT(0)$ space and $\Ga\acts X$ be
a geometric action.
Suppose $\F$ is a
$\Ga$--invariant collection of unbounded, closed, convex subsets.
Assume the following:
\begin{enumerate}
\renewcommand{\theenumi}{\textup{\Alph{enumi}}}
\item \label{item:AppAllFlats}
There is a constant $D<\infty$ such that
each flat $F\subseteq X$ lies in a $D$--tubular neighborhood of
some $C\in \F$.
\item \label{item:AppControlledInt}
For each positive $r<\infty$ there is a constant $\rho=\rho(r)<\infty$
so that for any two distinct elements $C, C' \in \F$ we have
\[
   \diam \bigl( \nbd{C}{r} \cap \nbd{C'}{r} \bigr) < \rho.
\]
\end{enumerate}
Then we conclude:
\begin{enumerate}
\item \label{item:AppPeriodic}
The collection $\F$ is locally finite,
there are only finitely many $\Ga$--orbits in~$\F$,
and each $C \in \F$ is $\Ga$--periodic.

\item \label{item:AppTitsBoundary}
Every connected component of $\tits X$ containing more
than one point is contained in $\tits C$ for a unique $C\in \F$.

\item \label{item:AppConeDirections}
Let $X_\omega$ be an asymptotic cone
$\Cone_\omega (X,\star_n,\lambda_n)$.
Let $\F_\omega$ denote the set of all subspaces $C_\omega \subseteq X_\omega$
of the form $C_\omega=\ulim C_n$ where $C_n \in \F$ and
$\ulim \lambda_n^{-1}\, d(C_n,\star_n) < \infty$.

Then for every $x\in X_\om$,
each connected component of $\Si_x X_\om$ containing more
than one point is contained in $\Si_x C_\om$ for
a unique $C_\om\in \F_\om$.  Furthermore,
if a direction $\oa{x y}$ lies in a nontrivial component of $\Si_x C_\om$
then an initial segment of $[x,y]$ lies in~$C_\om$.

\item \label{item:AppTreeGraded}
Every asymptotic cone $X_\om$ is tree-graded with respect
to the collection $\F_\om$.

\item \label{item:AppRelHyp}
$\Ga$ is hyperbolic relative to any collection $\mathcal{P}$
of representatives of the finitely many conjugacy classes of stabilizers of elements of $\F$.  

\item \label{item:AppBoundaryWellDefined}
Suppose the stabilizer of each $C\in \F$
is a $\CAT(0)$ group with very well-defined boundary.
Then $\Ga$ has a very well-defined boundary.
\end{enumerate}
\end{theorem}

In the sequel we will always assume that $X$, $\Gamma$ and $\F$
satisfy the hypotheses of Theorem~\ref{thm:Main}
(except in Lemma~\ref{lem:Ballmann} and Corollary~\ref{cor:Ballmann}).

\subsection{Acknowledgements}

The authors are grateful to Igor Belegradek for
many conversations about the original version of this article,
which led the authors to discover the gap in \cite{HK05}.
The first author is supported in part by NSF grant DMS-0731759.
The second author is supported in part by NSF grants DMS-0701515 and DMS-0805939.

\section{Proof of Theorem~\ref{thm:Main} assertions (\ref{item:AppPeriodic})
and (\ref{item:AppTitsBoundary})}

The proofs in this section are all easy modifications of arguments
from \cite{HK05}.

\begin{lem}[\emph{cf} Lemma~3.1.1, \cite{HK05}]
\label{lem:3.1.1}
The collection $\F$ is locally finite; in other words, only finitely
many elements of $\F$ intersect any given compact set.
\end{lem}

\begin{proof}
It suffices to show that only finitely many elements of $\F$
intersect each closed metric ball $\overline{B}(x,r)$.
Let $\F_0$ be the collection of all $C \in \F$ intersecting this ball.
By hypothesis (\ref{item:AppControlledInt}) of Theorem~\ref{thm:Main},
there exists $\rho=\rho(1)$ such that for any distinct elements
$C,C' \in \F$ we have
\[
   \diam \bigl( \nbd{C}{1} \cap \nbd{C'}{1} \bigr) < \rho.
\]
If we let $\kappa:=r+\rho$ then for each $C \in \F_0$ the set
$C \cap \overline{B}(x,\kappa)$ has diameter at least $\rho$
since $C$ is connected and unbounded.

If $\F_0$ is infinite then it contains a sequence of distinct
elements $(C_i)$ such that the compact sets $C_i \cap \overline{B}(x,\kappa)$
converge in the Hausdorff metric.
In particular, whenever $i$ and $j$ are sufficiently large,
the Hausdorff distance between $C_i \cap \overline{B}(x,\kappa)$
and $C_j \cap \overline{B}(x,\kappa)$ is less than $1$.
But $C_i \cap \overline{B}(x,\kappa)$ has diameter at least $\rho$
and lies in $\nbd{C_i}{1} \cap \nbd{C_j}{1}$,
contradicting our choice of $\rho$.
\end{proof}

\begin{proof}[Proof of Theorem~\ref{thm:Main}(\ref{item:AppPeriodic})]
The collection $\F$ is $\Gamma$--invariant by hypothesis
and is locally finite by Lemma~\ref{lem:3.1.1}.
Now \cite[Lemma~3.1.2]{HK05} implies that such a collection of subspaces
contains only finitely many $\Gamma$--orbits
and that each $C \in \F$ is $\Gamma$--periodic,
provided that each $C \in \F$ is a flat.
However the hypothesis that elements of $\F$ are flats is never used
in the proof of Lemma~3.1.2.
Thus the same conclusion holds in the present setting.
\end{proof}

The following three results were proved in \cite{HK05}
under the additional hypothesis that the elements of $\F$ are flats.
Again this hypothesis is never used in the proofs.

\begin{lem}[\emph{cf} Lemma~3.2.2, \cite{HK05}]
There is a decreasing function $D_1 = D_1(\theta) <\infty$
such that if $S \subset X$ is a flat sector of angle $\theta > 0$
then $S \subset \nbd{C}{D_1(\theta)}$ for some
$C \in \mathcal{F}$.\qed
\end{lem}

\begin{lem}[\emph{cf} Lemma~3.2.3, \cite{HK05}]
For all $\theta_0>0$ and $R<\infty$, there exist
$\delta_1=\delta_1(\theta_0,R)$ and $\rho_1=\rho_1(\theta_0,R)$
such that if $p,x,y \in X$ satisfy
$d(p,x), d(p,y) > \rho_1$ and
\[
   \theta_0 < \angle_p(x,y) \le \cangle_p(x,y)
   < \angle_p(x,y) + \delta_1 < \pi - \theta_0
\]
then there exists $C \in \mathcal{F}$ such that
\[
   \bigl([p,x] \cup [p,y]\bigr) \cap B(p,R)
   \subset \nbd{C}{D_1(\theta_0)}. \rlap{\hspace{1.29in}\qedsymbol}
\]
\end{lem}

\begin{prop}[\emph{cf} Proposition~5.2.1, \cite{HK05}]
\label{prop:5.2.1}
For each $\theta_0>0$ there is a positive constant
$\delta_4=\delta_4(\theta_0)$ such that whenever $p \in X$ and
$\xi,\eta \in \tits X$ satisfy
\begin{equation}
\label{eqn:CloseToTits}
\tag{$\dag$}
   \theta_0 \le \angle_p(\xi,\eta) \le \tangle(\xi,\eta)
   \le \angle_p(\xi,\eta) + \delta_4 \le \pi - \theta_0
\end{equation}
then there exists $C \in \mathcal{F}$ so that
\[
   [p,\xi] \cup [p,\eta] \subset \nbd{C}{D_1(\theta_0)}.
   \rlap{\hspace{1.65in}\qedsymbol}
\]
\end{prop}

\begin{proof}[Proof of Theorem~\ref{thm:Main}(\ref{item:AppTitsBoundary})]
The proof is essentially the same as for the
forward implication of \cite[Theorem~5.2.5]{HK05}.
By (\ref{item:AppControlledInt}), 
it is clear that if $C,C' \in \mathcal{F}$ are distinct
then $\tits C \cap \tits C' = \emptyset$.
If $\xi,\eta \in \tits X$ and
$0< \tangle(\xi,\eta) < \pi$
then we can find $\theta_0>0$ and $p \in X$ such that (\ref{eqn:CloseToTits})
holds for $\delta_4=\delta_4(\theta_0)$.
Hence by Proposition~\ref{prop:5.2.1}
we have $[p,\xi] \cup [p,\eta] \subset \nbd{C}{D_1(\theta_0)}$
for some $C \in \mathcal{F}$,
which means that $\{\xi,\eta\} \subset \tits C$.

More generally, suppose $\xi,\eta$ are distinct points in the same
component of $\tits X$.
Then there is a sequence $\xi=\xi_0,\dots,\xi_\ell=\eta$
such that $0< \tangle(\xi_i,\xi_{i+1}) <\pi$.
By the previous paragraph, it follows that $\{\xi,\eta\} \subset\tits C$
for some $C \in \mathcal{F}$.
\end{proof}

\section{Filling the gap}

The goal of this section is to prove Proposition~\ref{prop:AtMostOnePoint}
using new arguments not found in \cite{HK05}.
We will use the following result due to Ballmann.

\begin{lem}[Lemma~III.3.1, \cite{Ballmann95}]
\label{lem:Ballmann}
Let $X$ be any proper $\CAT(0)$ space.  Let $c$ be a geodesic line in $X$
which does not bound a flat strip of width $R>0$.
Then there are neighborhoods $U$ of $c(\infty)$ and $V$ of $c(-\infty)$
in $\overline{X}$ such that for any $\zeta \in U$ and $\eta \in V$
there is a geodesic from $\zeta$ to $\eta$, and for any such
geodesic $c'$ we have $d\bigl(c',c(0)\bigr) < R$.
\end{lem}

The following corollary of Ballmann's result
was observed by Hindawi in the setting of Hadamard manifolds.  

\begin{cor}[\emph{cf} Proposition~3.3, \cite{Hindawi}]
\label{cor:Ballmann}
Let $X$ be any proper $\CAT(0)$ space.
Suppose $p \in X$ and let $[x_n,y_n]$ be a sequence of geodesic
segments in $X$ such that $x_n$ and $y_n$ converge respectively to
$\xi_x$ and $\xi_y \in \boundary X$.
If $\Td(\xi_x,\xi_y) > \pi$ then the distances
$d\bigl(p,[x_n,y_n]\bigr)$ are bounded above as $n \to \infty$.
\end{cor}

\begin{proof}
If $\Td(\xi_x,\xi_y) > \pi$
then there exists a geodesic line $c$ in $X$ with endpoints $\xi_x$
and $\xi_y$ that does not bound a half-plane
(see for instance Ballmann \cite[Theorem~II.4.11]{Ballmann95}).
In particular, there exists $R$ such that $c$ does not bound a flat strip
of width $R$.
Once $n$ is sufficiently large, $x_n \in V$ and $y_n \in U$,
where $U$ and $V$ are the neighborhoods given by Lemma~\ref{lem:Ballmann}.
Therefore for all but finitely many $n$, we have
$d\bigl(c(0),[x_n,y_n]\bigr) < R$.
Thus
$d\bigl(p,[x_n,y_n]\bigr)$ remains bounded as $n\to\infty$.
\end{proof}

The next result shows that the convex hull of $C \cup C'$
lies within a uniformly bounded neighborhood of $C \cup C' \cup [p,q]$
where $[p,q]$ is any geodesic of shortest length from $C$ to $C'$.

\begin{prop}
\label{prop:NearShortestPath}
There is a constant $\epsilon_0>0$ such that the following holds.
Choose $C\ne C'$ in $\F$, and let $[p,q]$ be a geodesic of shortest length
from $C$ to $C'$.
Then every geodesic from $C$ to $C'$ comes within a distance $\epsilon_0$
of both $p$ and $q$.
\end{prop}

\begin{proof}
Suppose by way of contradiction that there were a sequence of counterexamples, ie, subspaces $C_i \ne C'_i$
in $\F$, points $p_i,x_i \in C_i$ and $q_i,y_i \in C'_i$
such that $[p_i,q_i]$ is a shortest path from $C_i$ to $C'_i$
and such that $d\bigl( p_i, [x_i,y_i] \bigr)$ tends to infinity.
We have two cases depending on whether $d(C_i,C'_i)$ remains bounded
as $i \to \infty$.

\emph{Case~1:}
Suppose $d(C_i,C'_i)$ remains bounded.
By Theorem~\ref{thm:Main}(\ref{item:AppPeriodic}), the $C_i$ lie in finitely many orbits.  Pass to a subsequence and translate by
the group action so that $C_i=C$ is constant.
Translating by $\Stab(C)$, we can also assume that $C'_i=C'$ is constant.
After passing to a further subsequence, the points $p_i$, $q_i$, $x_i$
and $y_i$ converge respectively to $p \in C$, $q \in C'$,
$\xi_x \in \boundary C$ and $\xi_y \in \boundary C'$.
Since $d\bigl(p,[x_i,y_i] \bigr)$ tends to infinity, it follows
from Corollary~\ref{cor:Ballmann} that $\Td(\xi_x,\xi_y) \le \pi$,
contradicting Theorem~\ref{thm:Main}(\ref{item:AppTitsBoundary}).

\emph{Case~2:} Now suppose the distances $d(C_i,C'_i)$ are unbounded.
After passing to a subsequence and applying elements of $\Gamma$,
we can assume that $C_i=C$ is constant and that the points 
$p_i$, $q_i$, $x_i$ and $y_i$ converge respectively to
$p \in C$, $\xi_q \in \boundary X$, $\xi_x \in \boundary C$
and $\xi_y \in \boundary X$.
Furthermore, $\xi_q \notin \boundary C$ since the ray from $p$
to $\xi_q$ meets $C$ orthogonally.
By hypothesis, $d\bigl( p, [x_i,y_i] \bigr)$ tends to infinity.
Since $d(C,C'_i) = d(p,C'_i) \to \infty$, we also have
$d\bigl( p,[y_i,q_i] \bigr) \to \infty$.
Therefore, by Corollary~\ref{cor:Ballmann}
the points $\xi_x$, $\xi_y$ and $\xi_q$ all lie in the same component
of $\tits X$,
contradicting Theorem~\ref{thm:Main}(\ref{item:AppTitsBoundary}).
\end{proof}

\begin{cor}\label{cor:FourPoint}
There is a constant $\epsilon_1$ such that the following holds.
Suppose $C\ne C' \in \F$ and we have $a,b \in C$ and
$a',b' \in C'$.
Then
\[
   d(a,b) + d(a',b') \le d(a,a') + d(b,b') + \epsilon_1.
\]
\end{cor}

\proof
Choose a geodesic $[p,q]$ of shortest length from $C$ to $C'$.
By Proposition~\ref{prop:NearShortestPath}, there are points
$x,x' \in [a,a']$ and $y,y' \in [b,b']$
such that $x$ and $y$ are within a distance $\epsilon_0$ of $p$
and $x'$ and $y'$ are within a distance $\epsilon_0$ of $q$.
Therefore
\begin{align*}
   d(a,b) + d(a',b')
     &\le d(a,x) + d(x,y) + d(y,b) + d(a',x') + d(x',y') + d(y',b') \\
     &\le d(a,a') + d(b,b') + 4\epsilon_0. \rlap{\hspace{2.13in}\qedsymbol}
\end{align*}

\begin{prop}\label{prop:AtMostOnePoint}
Suppose $C_\om,C'_\om \in \F_\omega$.
If $C_\om\ne C'_\om$ then $C_\om \cap C'_\om$ contains at most one point.
\end{prop}

\begin{proof}
Suppose $C_\om \ne C'_\om$.
Then $C=\ulim C_n$ and $C'_\om = \ulim C'_n$,
where $C_n \ne C'_n$ for $\om$--almost all $n$.
If $a,b \in C$, they are represented by sequences
$(a_n)$ and $(b_n)$ such that $a_n,b_n \in C_n$.
If $a,b$ are also in $C'$, they can also be represented by sequences
$(a'_n)$ and $(b'_n)$ with $a'_n,b'_n\in C'_n$.
Furthermore
\[
   \ulim \lambda_n^{-1} d(a_n,a'_n)
   = \ulim \lambda_n^{-1} d(b_n,b'_n)
   = 0.
\]
By Corollary~\ref{cor:FourPoint} we see that
\[
   d(a,b)
    = \ulim \lambda_n^{-1} d(a_n,b_n)
    \le \ulim \lambda_n^{-1} \bigl( d(a_n,a'_n)
        + d(b_n,b'_n) + \epsilon_1 \bigr)
    = 0.
\]
Thus $a=b$.
\end{proof}

\section{Proofs of Theorem~\ref{thm:Main} assertions
(\ref{item:AppConeDirections}), (\ref{item:AppTreeGraded})
and (\ref{item:AppRelHyp})}

The proofs in this section are modeled closely on arguments from
\cite{HK05}.
Indeed, the reader will not find any substantially new ideas in this
section.
However, in many places minor modifications are necessary to adapt the
proofs from the isolated flats setting to the present level of generality.
In these places we have provided the detailed arguments
for the benefit of the reader.

The proof of the following proposition is identical to that of
\cite[Proposition~3.2.5]{HK05}.

\begin{prop}\label{prop:3.2.5}
For all $\theta_0>0$ there are $\delta_2=\delta_2(\theta_0)>0$
and $\rho_2 = \rho_2(\theta_0)$ such that if $x,y,z \in X$, all vertex
angles and comparison angles of $\Delta(x,y,z)$ lie in
$(\theta_0,\pi-\theta_0)$, each vertex angle is within $\delta_2$ of the
corresponding comparison angle,
and all three sides of $\Delta(x,y,z)$ have length greater than $\rho_2$,
then
\[
   [x,y] \cup [x,z] \cup [y,z] \subset \nbd{C}{D_1(\theta_0)}
\]
for some $C \in \mathcal{F}$. \qed
\end{prop}

\begin{lem}[\emph{cf} Lemma~3.3.1, \cite{HK05}]
\label{lem:3.3.1}
For all $\theta_0 > 0$ there is a $\delta_3=\delta_3(\theta_0)>0$
such that if $x,y,z \in X_\om$ are distinct, all vertex angles and comparison
angles of $\Delta(x,y,z)$ lie in $(\theta_0,\pi-\theta_0)$, and each
vertex angle is within $\delta_3$ of the corresponding comparison
angle, then
\[
   [x,y] \cup [x,z] \cup [y,z] \subset C_\om
\]
for some $C_\om \in \F_\om$.
\end{lem}

\begin{proof}
The proof is essentially the same as the first part of
\cite[Lemma~3.3.1]{HK05}.
Choose $\theta_0$, and let $\delta_2$ and $\rho_2$ be the constants
provided by Proposition~\ref{prop:3.2.5}.
Set $\delta_3:=\delta_2$.
Choose $x,y,z \in X_\om$ as above,
and apply \cite[Corollary~2.4.2]{HK05} to get sequences
$(x_k)$, $(y_k)$ and $(z_k)$ representing $x$, $y$ and $z$
such that each vertex angle of $\Delta(x,y,z)$
is the ultralimit of the corresponding angle of $\Delta(x_k,y_k,z_k)$.
Then $\Delta(x_k,y_k,z_k)$ satisfies the hypothesis
of Proposition~\ref{prop:3.2.5} for $\om$--almost all $k$.
Consequently $x$, $y$ and $z$ lie in some $C_\om \in \F_\om$.
\end{proof}

\begin{proof}[Proof of Theorem~\ref{thm:Main}(\ref{item:AppConeDirections})]
The proof is similar to the proof of
\cite[Proposition~3.3.2]{HK05}.
If $\oa{x y}, \oa{x z} \in \Sigma_x X_\om$ and $0<\angle_x(y,z)<\pi$
then $\angle_x(y,z) \in (\theta,\pi-\theta)$ for some positive $\theta$.
Let $\delta_3=\delta_3(\theta/8)$ be the constant given by
Lemma~\ref{lem:3.3.1}, and let $\delta:=\min\{\delta_3,\theta/4\}$.
By \cite[Proposition~2.2.3]{HK05} there exist points $y' \in [x,y]$
and $z' \in [x,z]$ such that
the angles of $\Delta(x,y',z')$ are within $\delta$ of their
respective comparison angles
and also such that $d(x,y')=d(x,z')$.
Since $\delta\le \theta/4$, the angles of $\Delta(x,y',z')$ at $y'$ and $z'$
lie in the interval $(\theta/8,\pi/2)$.
Lemma~\ref{lem:3.3.1} now implies that
\[
   [x,y'] \cup [x,z'] \cup [y',z'] \subset C_\om
\]
for some $C_\om \in \F_\om$.
Thus the directions $\oa{x y}$ and $\oa{x z}$ both lie in
$\Sigma_x C_\om$ and any geodesic representing either direction
has an initial segment that lies in $C_\om$.
The uniqueness of $C_\om$ is an immediate consequence of
Proposition~\ref{prop:AtMostOnePoint}.

More generally, suppose $\oa{x y}$ and $\oa{x z}$ are distinct
directions in the same component of $\Si_x X_\om$.
Then there is a sequence $y=y_0,\dots,y_\ell=z$
such that $0 < \angle_x(y_{i-1},y_{i}) < \pi$ for $i=1,\dots,\ell$.
By the previous
paragraph, $\oa{x y_{i-1}}$ and $\oa{x y_{i}}$ both lie in
$\Sigma_x C_i$ for a unique $C_i \in \F_\om$,
and $[x,y_{i-1}]$ and $[x,y_{i}]$ have initial segments in $C_i$.
Since $[x,y_i]$ has initial segments in both $C_i$ and $C_{i+1}$,
it follows from Proposition~\ref{prop:AtMostOnePoint} that
$C_1=\cdots=C_{\ell-1}$.
\end{proof}

\begin{cor}[\emph{cf} Corollary~3.3.3, \cite{HK05}]
\label{cor:3.3.3}
Let $\pi_{C_\om} \colon X_\om \to C_\om$ be the nearest point projection.
Then $\pi_{C_\om}$ is locally constant on $X_\om \setminus C_\om$.
\end{cor}

\begin{proof}
Choose $s \in X_\om \setminus C_\om$, and let $x := \pi_{C_\om}(s)$.
Then $\angle_x(S,F)$ is at least $\pi/2$.
In particular, the direction $\oa{x s} \notin \Si_x C_\om$.
By continuity of $\log_x$,
if $U$ is any connected set containing $s$ in $X_\om \setminus \{x\}$
then $\log_x(U)$ is a connected set containing $\oa{x s}$.
Since $\log_x(U)$ is not contained in $\Si_x C_\om$, it follows from
Theorem~\ref{thm:Main}(\ref{item:AppConeDirections})
that $\log_x(U)$ is disjoint from $\Si_x C_\om$,
and that
each point of $\log_x(U)$ is at an angular distance $\pi$
from $\Si_x C_\om$.
Hence for each $s' \in U$, we have $\pi_{C_\om}(s') = x$.
\end{proof}

\begin{lem}[\emph{cf} Lemma~3.3.4, \cite{HK05}]
\label{lem:3.3.4}
If $p$ lies in the interior of the geodesic $[x,y] \subset X_\om$
and $x$ and $y$ lie in the same component of $X_\om \setminus \{p\}$
then $p$ is contained in an open subarc of $[x,y]$ that lies in
$C_\om$ for some $C_\om \in \F_\om$.
\end{lem}

\begin{proof}
By continuity of $\log_p$, the directions $\oa{px}$ and $\oa{py}$
lie in the same component of $\Si_x F_\om$, which is therefore nontrivial.
Theorem~\ref{thm:Main}(\ref{item:AppConeDirections})
now implies that initial segments of $[p,x]$ and $[p,y]$ lie in
$C_\om$ for some $C_\om \in \F_\om$.
\end{proof}

\begin{lem}[\emph{cf} Lemma~3.3.5, \cite{HK05}]
\label{lem:3.3.5}
Every embedded loop in $X_\om$ lies in some $C_\om \in \F_\om$. 
\end{lem}

\begin{proof}
Let $\gamma$ be an embedded loop containing points $x \ne y$.
For each $p \in [x,y]$, the loop $\gamma$ provides a path from
$x$ to $y$ that avoids $p$.
By Lemma~\ref{lem:3.3.4}, an open subarc of $[x,y]$ containing $p$
lies in some $C_p \in \F_\om$.
The interior of $[x,y]$ is covered by these open intervals.
By Proposition~\ref{prop:AtMostOnePoint}, it follows that $C_p=C_\om$
is independent of the choice of $p$.
Let $\beta$ be a maximal open subpath of $\gamma$ in the complement
of $C_\om$.  It follows from Corollary~\ref{cor:3.3.3}
that $\beta$ projects to a constant under $\pi_{C_\om}$.
Hence the endpoints of $\beta$ coincide, which is absurd.
\end{proof}

\begin{proof}[Proof of Theorem~\ref{thm:Main}(\ref{item:AppTreeGraded})]
Each $C \in \F$ is closed and convex in $X$.
Therefore each $C_\om \in \F_\om$ is closed and convex in $X_\om$.
By Proposition~\ref{prop:AtMostOnePoint}, distinct subspaces
$C_\om,C'_\om \in \F_\om$ intersect in at most one point.
Furthermore, Lemma~\ref{lem:3.3.5} implies that every embedded geodesic
triangle in $X_\om$ lies in some $C_\om \in \F_\om$.
\end{proof}

\begin{proof}[Proof of Theorem~\ref{thm:Main}(\ref{item:AppRelHyp})]
Let $\mathcal{P}$ be a set of representatives of the finitely many
conjugacy classes of stabilizers of elements of $\F$.
The action of $\Gamma$ on $X$ induces a quasi-isometry $\Gamma \to X$
that induces a one-to-one correspondence between
the left cosets of elements of $\mathcal{P}$ and
the elements of $\F$.
It follows from \Drutu--Sapir \cite[Theorem~5.1]{DrutuSapirTreeGraded}
that every asymptotic cone of $\Gamma$ is 
tree graded with respect to ultralimits of sequences of left cosets
of elements of $\mathcal{P}$.
Now
\cite[Theorem~1.11]{DrutuSapirTreeGraded} implies that $\Gamma$
is relatively hyperbolic with respect to $\mathcal{P}$.
\end{proof}


\bibliographystyle{gtart}
\bibliography{erratum}

\begin{thebibliography}

\bibitem{Ballmann95}
\textbf{W Ballmann}, \emph{Lectures on spaces of nonpositive curvature},
  volume~25 of \emph{DMV Seminar}, Birkh\"auser Verlag, Basel (1995), with an
  appendix by M Brin

\bibitem{DrutuSapirTreeGraded}
\textbf{C Dru{\RomanianComma{t}}u}, \textbf{M Sapir}, \emph{Tree-graded spaces
  and asymptotic cones of groups}, Topology 44 (2005) 959--1058, with an
  appendix by D Osin and M Sapir.

\bibitem{Hindawi}
\textbf{M~A Hindawi}, \emph{Large scale geometry of $4$--dimensional closed
  nonpositively curved real analytic manifolds}, Int. Math. Res. Not.  (2005)
  1803--1815

\bibitem{HK05}
\textbf{G\,C Hruska}, \textbf{B Kleiner}, \emph{Hadamard spaces with isolated
  flats}, Geom. Topol. 9 (2005) 1501--1538, with an appendix by M Hindawi, G\,C
  Hruska, and B Kleiner.

\end{thebibliography}

\end{document}